\documentclass[a4paper,11pt,oneside]{article}
\usepackage{amsmath,amsthm,amsfonts,amssymb,ifpdf}
\usepackage{amssymb}
\usepackage{amsmath}
\usepackage{amsthm}
\usepackage{graphicx}
\usepackage[all]{xy}
\usepackage{enumerate}
\usepackage{tikz-cd}
\usepackage{subcaption}
\usepackage{authblk}

\ifpdf        
\usepackage[
pdftex,
colorlinks,%
linkcolor=blue,citecolor=red,urlcolor=blue,
hyperindex,%
plainpages=false,%
bookmarksopen,%
bookmarksnumbered%
]{hyperref} 
\usepackage{thumbpdf}
\else
\usepackage{hyperref}
\fi

\usepackage{theoremref} 

\usepackage{tikz}
\usetikzlibrary{positioning, arrows.meta}

\newtheorem*{claim*}{Claim}

\newtheorem{theorem}{Theorem}[section]
\newtheorem{lemma}{Lemma}[section]

\newtheorem{proposition}{Proposition}[section]
\newtheorem{remark}{Remark}[section]
\newtheorem{definition}{Definition}[section]

\numberwithin{equation}{section}

\providecommand{\keywords}[1]
{
	\small	
	\textbf{\textit{Keywords---}} #1
}
\providecommand{\msc}[1]
{
	\small	
	\textbf{\textit{Mathematics Subject Classification (2020) ---}} #1
}

\title{{A family of potential wells for a semilinear pseudo-parabolic equation}}
\title{{Existence of global solutions to a semilinear pseudo-parabolic equation}}

\author{Joydev Halder\thanks{halderjoydev@gmail.com}\ }
\author[2]{Bhargav Kumar Kakumani\thanks{bhargav@hyderabad.bits-pilani.ac.in}}
\author{Suman Kumar Tumuluri \thanks{suman.hcu@uohyd.ac.in}}

\affil{School of Mathematics and Statistics, University of Hyderabad, Hyderabad, India}
\affil[2]{Department of Mathematics, BITS-Pilani, Hyderabad Campus, Hyderabad, India}

\begin{document}
	\maketitle

\begin{abstract}
In this article, we consider a  semilinear pseudo parabolic heat equation with the nonlinearity which is the product of logarithmic and polynomial functions. Here we prove the global existence of solution to the problem for arbitrary dimension $n \geq 1$ and power index $p>1$. Asymptotic behaviour of the solution has been addressed at different energy levels. Moreover, we prove that the global solution indeed decays with an exponential rate. Finally, sufficient conditions are provided under which blow up of solutions take place.
\end{abstract}
\\
\msc{35A01, 35B40, 35K61, 35S16}
\section{Introduction}
In this article, we are interested to study the global existence and the  longtime behaviour of the following semilinear pseudo parabolic heat equation
\begin{equation}\label{main}
	\left\{
	\begin{aligned}
		&v_t-\Delta v_t -\Delta v=v|v|^{p-1}\log|v|, && t\in \mathbb{R}^+,\ x \in {{U}}, 
		\\
		&v(x,t)=0, && t\in \mathbb{R}^+,\ x \in \partial {{U}}, 
		\\
		&v(x,0)=v_0(x), && x \in {{U}},
	\end{aligned}
	\right.
\end{equation}
where ${{U}} \subset \mathbb{R}^n$ $(n \geq 1)$ is a smooth bounded domain and the power index $p>1$.

Pseudo equations are characterized by the appearance of a Laplacian of a time derivative of the unknown. Specially, pseudo-parabolic equations describe various important physical phenomena, such as the heat conduction involving two temperatures \cite{Chen1968_19}, the seepage of homogeneous fluids in fissured rock \cite{Barenblatt1960_24},  unidirectional propagation of long waves in nonlinear dispersive media \cite{Benjamin1972_272, Ting1963_14},  the aggregation of populations \cite{padron2004_356}, etc. In particular, problem (\ref{main}) also describes several natural phenomena (see \cite{kaikina2007_2007, kaikina2010_44, matahashi1978_22,matahashi1979_22} and the references there in). For example, in the analysis of nonstationary processes in semiconductors  the presence of source term, $v_t- \Delta v_t$ denotes the free electron density rate, $\Delta v$ denotes the linear dissipation of free charge current, and the source term represents a source of free electron current (see \cite{korpusov2003_43}). On the other hand, partial differential equations (PDEs) with polynomial and logarithmic nonliniearities are studied widely due to many applications in physics and other applied science such as theory of superfluidity, nuclear physics, transport phenomenon and diffusion phenomenon (see \cite{chen2012_3,chen2015_258,gazzola2005_18,chao2016_437,lian2020_40,yacheng2006_64,zloshchastiev2010_16}).

Let us consider the following pseuudo parabolic equation
\begin{equation}\label{pseudo_parabolic_equation}
	\left\{
	\begin{aligned}
		&v_t-\Delta v_t-\Delta v=f(v), &&  t\in \mathbb{R}^+,\ x \in {{U}},
		\\
		&v(x,t)=0, && t\in \mathbb{R}^+,\ x \in \partial {{U}}, 
		\\
		&v(x,0)=v_0(x), && x \in {{U}}.
	\end{aligned}
	\right.
\end{equation}
The authors in \cite{ting1969_21} considered \eqref{pseudo_parabolic_equation} when $f=0$ and proved the existence and uniqueness of solutions. After this precursory result,  many authors investigated the global exsistence, uniqueness, longtime behavior, regularity and blow-up phenomena of the solutions to \eqref{pseudo_parabolic_equation} for different choices of $f$ using various methods (see \cite{ambrosetti1973_14,hoshino1991_34,stetter1973_23,tsutsumi1972_17,weissler1979_32} and the references therein). In \cite{ payne1975_22,sattinger1968_30}, Sattinger proposed a powerful technique (popularly known as the``potential well method'') to study the existence and longtime behaviour of the solution to \eqref{pseudo_parabolic_equation}. Sattinger and Payne's work \cite{payne1975_22} was the most influential one which was followed by many mathematicians. In \cite{payne1975_22}, the authors  introduced a potential well $W$, an outer potential well $V$ in terms of a Nehari functional $I(u)$ and a  potential energy functional $J(v)$.  Under few technical assumptions on $f$, they studied the properties of $W,\ V,\ J(v) $. Using the invariance of the potential well,  they proved the finite time blow-up of the solutions to the nonlinear heat equation. After that, the method was improved by many mathematicians and physicists to study different types of PDEs such as nonlinear heat, and nonlinear wave equations (see \cite{chen2012_3,chen2015_258,yacheng2003_192,liu2020_28,lian2020_40,yacheng2006_64,han2019_164,chen2015_422,han2019_474}). The authors of \cite{xu2013_264} considered \eqref{pseudo_parabolic_equation} when $f(v)=v^p$ and under some assumptions on $p$ and $n$ they  established existence and studied asymtoptic behaviour of the solution.  Using the properties of potential wells (PWs) the authors in \cite{xu2013_264}  proved the global existence and investigated the longtime behavior of solutions. However in \cite{liu2018_274, xu2016_270}, the authors amend the proofs of  \cite{xu2013_264}.  In particular, the authors  of \cite{zhou2021_200} generalized and extended the results of  \cite{xu2013_264} using a family of PWs. Moreover, the authors of \cite{chen2015_258} investigated the global existence,  exponential decay and finite time blow-up of solutions to \eqref{pseudo_parabolic_equation} when the source term is given by $v\log|v|$. Recently, many authors  studied \eqref{pseudo_parabolic_equation} using method of potetial wells for different choice of $f(v)$ (see \cite{ liu2018_2018,luo2015_38, xu2020}). In particular, the authors of \cite{zhou2020_1} considered \eqref{pseudo_parabolic_equation} and studied the global existence, blow-up of solution when the source term is given by $|x|^\sigma v|v|^{p-1}$,  $1<p<\infty, \ \sigma> -n$, if $n=1,2$; $1 <p \leq  \frac{n+2}{n-2}, \sigma > \frac{(1+p)(n-2)}{2}-n$, if $n\geq 3$. 

 Moreover, when $f(v)$ in \eqref{pseudo_parabolic_equation} is solely polynominal type $i.e.$, $f(v)=|v|^p$ or $v|v|^{p-1}$, then the authors of \cite{han2020_99} established the global existence, and blow up with restrictions on the power index  viz.,   $1 <p \leq  \frac{n+2}{n-2}$ if $n\geq 3$. The main objective of this paper is to study the global existence, and blow up phenomenon of solutions to \eqref{main} with out any restriction on $p$ and $n$. We also discuss the asymptotic behavior of solution to \eqref{main} at all three energy levels (super-critical, critical and sub-critical). 

The paper is organized as follows. We give some definitions and introduce the family of  PWs in Section 2. We recall the results which are helpful to prove the global existence and longtime behavior of the solution to \eqref{main} in Section 2. In Section 3, we prove the global existence of a weak solution to the problem \eqref{main} under some conditions for an arbitrary $n$ and $p$. A decay estimate on the solution in $H_0^1$-norm is discussed in Section 4. In the last section, we obtian the finite time blow-up of weak solutions to \eqref{main} under the cases $J(v_0) < d$, $J(v_0) = d$ and $J(v_0) > 0$.
\section{Preliminaries}
Henceforth, we denote the norm $\|\cdot\|_p$ by $\|g\|_p = (\int_{{U}} |g|^pdx )^{\frac{1}{p}}, \forall g \in L^p({{U}})$, $1 \leq p < \infty$ and the inner product $(h,g)_2=\int_{{U}} h g dx$. When $p=2$, we simply write $\|\ \|$ instead of $\| \ \|_2$.
\begin{definition} (Weak solution)
	A function 
	\[
	v \in L^\infty ((0,T); H^1_0({{U}}))\ \ {\rm with\ } v_t \in  L^2 ((0,T); L^2 ({{U}})), v(x,0)=v_0(x)
	\]
 is called a weak solution to \eqref{main} if it satisfies \eqref{main} in weak sense $i.e.,$ \begin{equation}\label{weaksolution}
 	{(v_t, w)_2+(\nabla v_t, \nabla w)_2+(\nabla v, \nabla w)_2=(v|v|^{p-1}\log|v|, w)_2, \forall w\in H^1_0({{U}}),\ t\in(0,T).} 
 \end{equation}
\end{definition}
\noindent
For completeness, we recall here the   definitions of the maximal existence time, and the notion of the finite time blow-up which are quite standard.
\begin{definition}
	Let $v:=v(x,t)$ be a weak solution to \eqref{main}. The maximal existence time $T$ of $v$ is defined as follows:\\
	$(i)$ If $v$ exists for every $t >0$, then $T=\infty$.
	\\
	$(ii)$ If there exists $\hat{t} \in (0,\infty)$ such that $ v$ exists for $0\leq t <\hat{t}$, but does not exist at $t=\hat{t}$, then $T=\hat{t}$.
\end{definition}
\begin{definition}
	A weak solution $v$ of \eqref{main} is said to blow-up in finite time if the maximal existence time $T$ is finite and $$\lim\limits_{t \to T^-} \|v(.,t)\|_{H^1_0({{U}})}=\infty.$$
\end{definition}
\subsection{Potential wells}
In this subsection, we define the energy functionals  associated to the nonlinear term $v|v|^{p-1}\log |v|$ and the corresponding family  of PWs and state few  results which are useful for analysis in the subsequent sections.\\
First, we define the potential energy functional $J$ as
\begin{equation}\label{potential}
	J(v)=\frac{1}{2}\|\nabla v\|^2- \frac{1}{1+p}\int\limits_{{{U}}}|v|^{1+p}\log |v| dx+\frac{1}{(1+p)^2}\|v\|_{1+p}^{1+p},
\end{equation}
and the Nehari functional $I$ is defined as
\begin{equation}\label{nehari}
	I(v)=\|\nabla v\|^2- \int\limits_{{{U}}}|v|^{1+p}\log |v| dx.
\end{equation}
From \eqref{potential} and \eqref{nehari}, observe that 
\begin{equation}\label{combo}
	J(v)=\frac{p-1}{2(1+p)}\|\nabla v\|^2+\frac{1}{(1+p)^2}\|v\|_{1+p}^{1+p} +\frac{1}{1+p}I(v).
\end{equation}
The Nehari manifold $\mathcal{N}$ is defined as 
\begin{equation*}
	\mathcal{N}(v)=\{v \in H^1_0({{U}}): I(v)=0\ {\rm and\ } \|\nabla v\|^2 \neq 0\},
\end{equation*}
and the depth of the well $d$ is defined as
\begin{equation*}
	d=\inf\limits_{v\in \mathcal{N}} J(v).
\end{equation*}
We now introduce the potential well $W$ and the outer potential well $V$ as follows:
\begin{equation*}
	W= \{0\} \cup \{v\in H_0^1({{U}}):   I(v)>0,\ J(v)<d\},
\end{equation*}
and
\begin{equation*}
	V=\{v\in H_0^1({{U}}):  J(v)<d,\ I(v)<0 \}.
\end{equation*}
If $v$ is a weak solution to \eqref{main} then multiplying $v_t$ to \eqref{main} and integrating over ${{U}} \times [0,t)$, we get 
\begin{equation}\label{energy_inequality}
	\int\limits_{0}^{t} \|v_t(\cdot,\tau)\|^2_{H^1_0({{U}})} d\tau +J(v(.,t))=J(v_0(.)),\  t\in [0,T).
\end{equation}
Set the energy functional $E$ associated to \eqref{main} by
\begin{equation*}
	E(v(\cdot,t))=J(v(\cdot,t))+ \int\limits_{0}^t \|v_t\|^2_{H^1_0({{U}})} d \tau.
\end{equation*}
The equation \eqref{energy_inequality}  immediately gives $E(v(\cdot,t)) = E(v_0)$. Thus \eqref{energy_inequality} is referred as the conservation of energy.\\
Next, we extend the notion of a ``single potential well to the family of PWs'' by defining
\begin{equation}\label{d_potential}
	J_{\delta}(v)=\frac{\delta}{2}\|\nabla v\|^2- \frac{1}{1+p}\int\limits_{{{U}}}|v|^{1+p}\log |v| dx+\frac{1}{(1+p)^2}\|v\|_{1+p}^{1+p},
\end{equation}
and 
\begin{equation}\label{d_nehari}
	I_{\delta}(v)=\delta \|\nabla v\|^2- \int\limits_{{{U}}}|v|^{1+p}\log |v| dx,
\end{equation}
where $\delta >0$.
We define the corresponding Nehari manifolds $\mathcal{N}_{\delta}$ as
\begin{equation*}
	\mathcal{N}_{\delta}(v)=\{v \in H^1_0({{U}}): I_{\delta}(v)=0,\ \|\nabla v\|^2 \neq 0\},
\end{equation*}
and ``the depth of family of PWs'' $d(\delta)$ as
\begin{equation}\label{d_delta}
	d(\delta)=\inf\limits_{v\in \mathcal{N}_\delta} J(v).
\end{equation}
Moreover, we define the family of PWs $W_\delta$ as
\begin{equation*}
	W_\delta=\{v\in H_0^1({{U}}): J(v)<d(\delta),\ I_\delta (v)>0 \} \cup \{0\},
\end{equation*}
and the outer of the family of PWs $V_\delta$ as
\begin{equation*}
	V_\delta=\{v\in H_0^1({{U}}): J(v)<d(\delta),\ I_\delta(v)<0 \}.
\end{equation*}
Since the potential energy $J$ and $I$ defined in  \eqref{potential}--\eqref{nehari} are the same as those in \cite{lian2020_40} (in the context of semilinear wave equation). Now, ``we recall few results from \cite{lian2020_40} and \cite{chen2015_258}:
\begin{lemma}(Cf. \cite{lian2020_40}, Lemma 2.1)\thlabel{lemma1}
	Let $v \in H_0^1 ({{U}})$ and $\|v\|\neq 0$. Then the following hold true.\medskip \\
	$(i)$ $\lim\limits_{\beta \to 0} J(\beta v)=0$, $\lim\limits_{\beta \to \infty} J(\beta v)=- \infty.$\medskip \\
	$(ii)$ There exists $\beta^* = \beta ^*(v)$ in $(0, \infty)$ such that
	$$\frac{d}{d \beta} J(\beta v) \mid_{\beta= \beta^* }=0.$$
	$(iii)$ The function $\beta \mapsto J(\beta v)$ is decreasing on $\beta^* \leq \beta <\infty$, increasing on $0 \leq \beta \leq  \beta^*$ and attains its maximum at $\beta=\beta^*$.\medskip \\
	$(iv)$ The function $I$ satisfies  $I(\beta^*v)=0$, $I(\beta v)=\beta \frac{d}{d \beta}J(\beta v) >0$ for $0<\beta <\beta^*$, and $I(\beta v)< 0$ for $\beta^* < \beta < \infty$. 
\end{lemma}
\begin{lemma}(Cf. \cite{lian2020_40}, Lemma 2.2)\thlabel{lemma2}
	Assume $\delta>0$ and  $\phi(r)=C^{p+2}r^p,$ where $C=\sup\limits_{v\in H^1_0({{U}})} \displaystyle\frac{\|v\|_{p+2}}{\|\nabla v\|}$ and $r(\delta)$ is the unique real root of the equation $\phi(r)=\delta$.  Then we have:
	\medskip \\
	$(i)$ $I_\delta(v)>0$, provided $0<\|\nabla v\| \leq r(\delta)$.
	\medskip \\
	$(ii)$ $\|\nabla v\|> r(\delta)$, provided $I_\delta(v)<0$.
	\medskip \\
	$(iii)$ $\|\nabla v\|> r(\delta)$ or $\|\nabla v\|=0$, provided $I_\delta(v)=0$. 
	
\end{lemma}

\begin{lemma}(Cf. \cite{lian2020_40}, Lemma 2.3)\label{lemma3}
	The function $\delta \mapsto d(\delta)$ defined in \eqref{d_delta} has the following properties:\medskip \\
	$(i)$ $d(\delta)= (\frac{1}{2}-\frac{\delta}{1+p})r^2(\delta)>0$ for $0<\delta <\frac{1+p}{2}$.
	\medskip \\
	$(ii)$ There exists a unique $ \delta_0>\frac{1+p}{2}$ such that $d(\delta_0)=0$, and $d(\delta)>0$ for $0<\delta<\delta_0$.
	\medskip \\
	$(iii)$ The function $\delta \mapsto d(\delta)$ is decreasing on $1\leq \delta \leq \delta_0$, strictly increasing on $0 <\delta \leq 1$, at $\delta=1$ this function attains local maximum and $d(1)=d$. 
\end{lemma}
\begin{lemma}(Cf. \cite{lian2020_40}, Lemma 2.4)\label{lemma_inv}
	Assume $0<J(v)<d$ for some $v \in H^1_0({{U}})$ and $\delta_1, \delta_2$ are two roots of the equation $d(\delta)=J(v)$ such that $\delta_1< 1< \delta_2$, then the sign of $I_\delta(v)$ does not change in $\delta_1<\delta < \delta_2$. 
\end{lemma}
\begin{theorem}[Invariant sets](Cf. \cite{halder2022_decay}, Theorem 2.1)\label{invariant}
	Let $v_0 \in H^1_0({{U}})$ and  $0<\eta<d$. Assume that $\delta_1<\delta_2$ are the two roots of equation $d(\delta)=\eta$, then:\medskip  \\
	$(i)$ All solutions to problem \eqref{main} with $0<J(v_0)\leq \eta$ belong to $W_\delta$ for $\delta_1<\delta<\delta_2$, provided $I(v_0)>0$ or $\|\nabla v_0\|=0$.
	\medskip \\
	$(ii)$ All solutions to problem \eqref{main} with $0<J(v_0)\leq \eta$ belong to $V_\delta$ for $\delta_1<\delta<\delta_2$, provided $I(v_0)<0$. 
\end{theorem}

\begin{proposition}(Cf. \cite{chen2015_258}, Theorem 2.5)
	\label{sign_I}
	Let $v_0\in H_0^1({{U}})$, $0<J(v_0)\leq d$ and $v(x,t)$ is a weak solution to the problem \ref{main}. Then we have the following conclusions:
	\begin{enumerate}
		\item If $I(v_0)>0$ then $I(v)>0,\ \quad \forall t\in [0,T)$
		\medskip
		\item If $I(v_0)<0$ then $I(v)<0,\ \quad \forall t\in [0,T),$
	\end{enumerate}
	where $T$ is the maximal existence time of $v$. 
\end{proposition}

\noindent The proof's of the above results follows from the similar lines given in \cite{lian2020_40} and \cite{chen2015_258}. So we omit the details.''

%
\section{Global existence}
In this section, we prove global existence of solution to \eqref{main} when $I(u_0) \geq 0$.  We use the Galarkin method to find a sequence of approximate solution to \eqref{main}. Using compactness arguments, we can find a subsequence of this approximate solution whose limit turns out to be a global solution to \eqref{main}. The details are given in the following theorems.

\begin{theorem}\label{global_exis}
	Let $v_0\in H_0^1({{U}})$. Assume $I(v_0)\geq 0$ and $J(v_0)\leq d$, then there exists a  global weak solution $v$ to problem \eqref{main} with $v\in L^{\infty}(0,\infty;H_0^1({{U}}))$ and $v_t \in L^{\infty}(0,\infty;L^2({{U}}))$.
\end{theorem}

\begin{proof}
	We divide the proof into two cases. In Case 1, we consider the situation where $I(v_0)\geq 0$, and $J(v_0)<d$ and the next case deals with $J(v_0)=d$ and $I(v_0)\geq 0$. \medskip  \\
	\textbf{Case $1$:} $I(v_0)\geq 0$, $J(v_0)<d$. \medskip\\
	\noindent We have the following possibilities:
	\begin{enumerate}
		\item[(a)] Suppose $I(v_0)\geq 0$ and $J(v_0)<0$. Then we get a contradiction to \eqref{combo}. Therefore this situation does not arise.\medskip
		\item[(b)] Assume that $I(v_0)\geq 0$ and $J(v_0)=0$. From \eqref{combo} it readily implies that $v_0=0$. Thus $v\equiv 0$ is a global solution to \eqref{main}.\medskip
		\item[(c)] If $0<J(v_0)<d$ and  $I(v_0)=0$, then we have $J(v_0)\geq d$ (from the definition of $d$) which is a contradiction.\medskip
	\end{enumerate}
\noindent Hence the remaining possibility in this case is $I(v_0)> 0$ and $0<J(v_0)<d$ and we focus only  on this possibility. Let $\{\psi_i\}$ be an orthogonal basis of $H_0^1({{U}})$. Let $V_m={\rm span}\{ \psi_1,\psi_2,...,\psi_m \}$ and the projection of the initial data on $V_m$ be given by
\[
v_m(x,0) =  \sum\limits_{j=1}^m a_j\psi_i(x),
\]
$i.e.,$
\begin{equation}\label{sequence_v_0}
v_m(.,0) \rightarrow v_0(.),\ \ \ {\rm in}\ H_0^1({{U}}).
\end{equation}
We now seek a sequence of functions $(v_m)$ of the form
\[
v_m(x,t) = \sum\limits_{j=1}^m q_j^m(t)\psi_j(x),\ \ m=1,2,3,...,
\]
where each $v_m$ satisfies the following approximated problem
\begin{equation}\label{app_sol}
\left\{
\begin{aligned}
&\Big( \frac{\partial v_{m}}{\partial t},\psi \Big)_{2} + \Big( \frac{\partial \nabla v_{m}}{\partial t} ,\nabla \psi \Big)_{2} + \big(\nabla v_{m},\nabla \psi \big)_{2} = \big( v_{m}|v_m|^{p-1}\log|v_m|,\psi \big)_{2},\  \psi\in V_m,
\medskip \\
&v_m(x,0) =  \sum\limits_{j=1}^m (v_0,\psi_j)\psi_j, \medskip \\
& q_j^m(0)=a_j.
\end{aligned}
\right.
\end{equation}
Observe that \eqref{app_sol} is a system of nonlinear ordinary differential equations (ODEs) for the unknown functions $q_j^m(t)$. Using the standard existence theory of ODEs, we obtain
\[
q_j:[0,t_m) \rightarrow \mathbb{R},\ \ \ \ j=1,2,3,...,m, 
\]
and $q_j$ satisfy \eqref{app_sol} in a maximal interval $[0,t_m)$, where $t_m\in (0,T]$.  We next show that for sufficiently large $m$, $v_m(.,t)\in W$ for $\ 0<t<t_m.$ For, we first notice that $v_m(.,0)\in W$ for sufficiently large $m$. If $v_m(.,\tilde{t})\in \partial W$ for some $\tilde{t}$, then we have either $v_m(.,\tilde{t})\in \mathcal{N}$ or $J(v_m(.,\tilde{t})) = d$. In both cases we get $J(v_m(.,\tilde{t}))\geq d$, which is contradiction. This shows that there exists $M\in \mathbb{N}$ such that $v_m(.,t)\in W$ for $m\geq M,\ t\in [0,t_m)$.\\
On the other hand, on multiplying equation \eqref{app_sol} with $q_j'(t)$, taking summation from $1$ to $m$, and integrating over $t$, we get
\begin{equation}
\label{int_formulation}
\displaystyle\int_0^t \Big[ \|\frac{\partial v_{m}}{\partial t}\|^2 + \|\frac{\partial \nabla v_{m}}{\partial t}\|^2 \Big] dt + J(v_m(.,t))=J(v_m(.,0)),\ m\in \mathbb{N},\ t\in (0,t_m).
\end{equation}
This readily implies $J(v_m(.,t))< d$. Note that
\begin{equation*}
\begin{aligned}
J(v_m(.,t))&= \frac{1}{2}\|\nabla v_m\|^2- \frac{1}{1+p}\int\limits_{{{U}}}|v_m|^{1+p}\log |v_m| dx+\frac{1}{(1+p)^2}\|v_m\|_{1+p}^{1+p}
\\
&\geq \frac{1}{1+p}I(v_m) + \Big( - \frac{1}{1+p} + \frac{1}{2} \Big) \| \nabla v_m \| +\frac{1}{(1+p)^2}\|v_m\|_{1+p}^{1+p} 
\\
&\geq \frac{p-1}{2p+2} \|\nabla v_m \| +\frac{1}{(1+p)^2}\|v_m\|_{1+p}^{1+p}.
\end{aligned}
\end{equation*}
Hence we find
\begin{equation}
\label{est_grad_u}
	\|\nabla v_m\| < \frac{2(1+p)}{p-1}J(v_m)\leq  \frac{2(1+p)d}{p-1},
\end{equation}
and
\begin{equation}
\label{est_norm_u_p}
\|v_m\|_{1+p}^{1+p}  \leq (1+p)^2J(v_m) \leq (1+p)^2d.
\end{equation}
From \eqref{int_formulation} and using $J(v_m)>0$, we have
\[
\int_0^t \Big[ \|\frac{\partial v_{m}}{\partial t}\|^2 + \|\frac{\partial \nabla v_{m}}{\partial t}\|^2 \Big] dt  < d.
\]
So the sequence of approximate solutions is uniformly bounded and it is independent of $m$ and $t$. Thus, we can extend $q_m$ to $[0,T)$. Let $\gamma = \frac{2p+2}{2p+1}$, observe that 
\begin{equation}
	\label{est_log_term}
	\begin{aligned}
		\int_{{{U}}}\big(|v_m|^{p}|\log|v_m\|\big)^{\gamma}dx
		&=  \Bigg[ \int\limits_{\{x\in {{U}}:\ v_m \leq 1\}} + \int\limits_{\{x\in {{U}}:\ v_m> 1\}} \Bigg] \big(|v_m|^{p}|\log|v_m\|\big)^{\gamma}dx, \medskip \\
		&\leq (ep)^{-\gamma}|{{U}}| + 2^{\gamma}\int_{{{U}}} |v_m|^{\gamma(p+\frac{1}{2})}dx, \medskip \\
		&= (ep)^{-\gamma}|{{U}}| + 2^{\gamma}\int_{{{U}}} |v_m|^{1+p}dx, \medskip \\
		& \leq (ep)^{-\gamma}|{{U}}| + 2^{\gamma}(1+p)^2d.
	\end{aligned}
\end{equation}
where the last inequality follows from \eqref{est_norm_u_p}.
 Moreover, we obtain
\begin{equation}\label{uni_bdd}
\left\{
\begin{aligned}
&v_m\  {\rm is\ uniformly\ bounded\ in}\ L^{\infty}((0,T);H_0^1({{U}})),
\\ 
& \frac{\partial v_{m}}{\partial t}\ {\rm is\ uniformly\ bounded\ in}\ L^{2}((0,T);H^1_0({{U}})),
\\
& v_m|v_m|^{p-1}\log|v_m| {\rm is\ uniformly\ bounded\ in }\ L^{\infty}((0,T);L^{\gamma}({{U}})).
\end{aligned}
\right.
\end{equation}
From \eqref{uni_bdd}, it is easy to see that there exists a subsequence of $(v_m)$ which is still denoted by $(v_m)$ and $v$ such that
\begin{equation}
\label{log_term}
\left\{
\begin{aligned}
	&v_m \xrightarrow{w^*} v\ {\rm in}\ L^{\infty}((0,T);H_0^1({{U}})), 
	\\
	&\frac{\partial v_m}{\partial t} \xrightarrow{w^*} \frac{\partial v}{\partial t}\ {\rm in}\ L^{2}((0,T);H^1_0({{U}})),
	\\
	&v_m|v_m|^{p-1}\log|v_m| \xrightarrow{w^*} v|v|^{p-1}\log|v|\ {\rm in}\ L^{\infty}((0,T);L^{\gamma}({{U}})).
\end{aligned}
\right.
\end{equation}

In view of \eqref{log_term} and using the standard arguments, we can show that $v$ satisfies the weak formualtion \eqref{app_sol}. Repeating the same argument from $[T,2T],\ [2T,3T],....$ we obtain that $v$ is a global weak solution  to \eqref{main}, and $v \in W$. \medskip \\
\textbf{Case $2$:} Denote $v_{0m}=\mu_m v_0$ where $\mu_m=1-\frac{1}{m},\ m\geq 2$. Consider problem \eqref{main} with initial data $v(x,0)= v_{0m}(x)$ i.e.,
\begin{equation}\label{main_approx}
	\left\{
	\begin{aligned}
		&w_t-\Delta w_t -\Delta w=w|w|^{p-1}\log|w|, && t\in \mathbb{R}^+,\ x \in {{U}}, 
		\\
		&w(x,t)=0, && t\in \mathbb{R}^+,\ x \in \partial {{U}}, 
		\\
		&w(x,0)=v_{0m}(x), && x \in {{U}},
	\end{aligned}
	\right.
\end{equation}
 Now observe that 
\[
J(v_{0m}) = J(\mu_m v_0) < J(v_0) =d \ \ \ {\rm and}\ \ I(v_{0m})>0.
\]
Following the similar steps from Case $1$, there exists a global solution to \eqref{main_approx}, say $w_m$. Moreover, we have
\begin{equation}
\label{int_formulation1}
\displaystyle\int_0^t \Big[ \|\frac{\partial w_{m}}{\partial t}\|^2 + \|\frac{\partial \nabla w_{m}}{\partial t}\| \Big] dt + J(w_m(.,t))=J(w_{m}(.,0))<d.
\end{equation}
Also, by the invariance of $W$, we get $I(w_m(.,t))>0$. Therefore there exists a subsequence of ${w_m},$ which is still denoted with ${w_m},$ such that 
$$w_m \xrightarrow{w^*} w \ {\rm in}\ L^{\infty}((0,T);H_0^1({{U}})),$$
$$ \frac{\partial w_m}{\partial t} \xrightarrow{w^*} \frac{\partial w}{\partial t}\ {\rm in}\ L^{2}((0,T);H^1_0({{U}})),$$
for each $T>0$. Using the arguments presented in Case 1, we conclude that $w\in W$ is a global weak solution to \eqref{main}. Hence the theorem is proved.
\end{proof}
Next we discus the global existence of the solutions to \eqref{main} under the following restriction on $p$ and dimension $n$
\begin{equation}\label{p_restriction}
	\begin{aligned}
		1<p< 
		\begin{cases}
			\infty, & \text{if } n \leq 2,\\
			\frac{4}{n-2}, & \text{if } 3 \leq n \leq 5.
		\end{cases}
	\end{aligned}
\end{equation}
For, we define \\
\medskip
\centerline{$\mathcal{N}_\alpha=\{v\in \mathcal{N} : J(v)< \alpha\},$}
\\ \medskip
\centerline{$\Lambda_\alpha= \inf\limits_{v \in \mathcal{N}_\alpha} \|v\|^2_{H_0^1({{U}})},$}
\\ \medskip
where $\alpha>d$.
\begin{remark}
	For any $\alpha> d$, if power index $p$ satisfies \eqref{p_restriction} and $\|\nabla v\| \neq 0$, then $\Lambda_\alpha >0$.
\end{remark}
\begin{proof}
	Let $v \in \mathcal{N}_\alpha$ and $\|\nabla v\| \neq 0$. From the definition of $I$, we get
	\begin{align}
		\|\nabla v\|^2=&\int\limits_{{{U}}}|v|^{1+p}\log |v| dx \nonumber
		\\
		=& \|v\|^{p+2}_{p+2} \nonumber
		\\
		=& C \| \nabla v \|^{p+2}, \label{new_bound_alpha}
	\end{align}
	where $C$ is the Sobolev embedding constant. From \eqref{new_bound_alpha}, one can easily get $\|v \|^2_{H_0^1({{U}})} >(\frac{1}{C})^{\frac{2}{p}}$. This completes the proof.
\end{proof}
\begin{theorem}\label{global_exis_super_critical}
	Let $v_0\in H_0^1({{U}})$ and power index $p$ satisfies $\eqref{p_restriction}$. For a given $\alpha >d$, assume that  $J(v_0) < \alpha$, $\|v_0 \|^2_{H_0^1({{U}})} < \Lambda_\alpha$ and $I(v_0) > 0$, then there exists a  global weak solution $v$ to problem \eqref{main} with $v\in L^{\infty}(0,\infty;H_0^1({{U}}))$ and $v_t \in L^{\infty}(0,\infty;L^2({{U}}))$.
\end{theorem}
\begin{proof}
	Using the same argument employed in Theorem \ref{global_exis}, we obtain a sequence $v_m$ which satisfies \eqref{app_sol}. From \eqref{sequence_v_0}, it is clear $I(v_m(\cdot,0)) >0$, for sufficiently large $m$. Next we show $I(v_m(\cdot,t)) >0$, for sufficiently large $m$. On the contrary, we assume that there exists a $\tilde{t}>0$ such that  $I(v_m(\cdot,\tilde{t})) =0$ and  $I(v_m(\cdot,t)) >0$, $0 \leq t < \tilde{t}$.
	From the fact
	\begin{equation}\label{global_U_derivative}
		\frac{d}{dt} \|v_m(\cdot, t)\|_{H^1_0({{U}})}^2=-2I(v_m(\cdot, t)),
	\end{equation}
	we deduce that $ t \mapsto \|v_m(\cdot, t)\|_{H^1_0({{U}})}^2$ is a decreasing function, for $0 \leq t < \tilde{t}$. Hence, we conclude that 
	\begin{equation}\label{inequality_lambda}
		\|v_m(\cdot, \tilde{t})\|_{H^1_0({{U}})}^2 <\|v_m(\cdot, 0)\|_{H^1_0({{U}})}^2<\Lambda_\alpha.
	\end{equation}
	On the other hand,  \eqref{energy_inequality} give us $J(v_m(\cdot, \tilde{t})) < J(v_m(\cdot, 0))$, which is a contradiction to \eqref{inequality_lambda}. Therefore for sufficiently large $m$, we obtain $I(v_m(\cdot,t)) >0, t\geq 0$. 
	The rest of the proof follows along the same lines of that Theorem \ref{global_exis} and the  only difference is $d$ needs to be replaced by $\alpha$.
\end{proof}
\section{Decay estimate}
In this section, we prove few decay estimates of the global  solutions to \eqref{main} in the $H_0^1$-norm whenever $J(v_0)\leq d$, $I(v_0) >0$. As in the proof of Theorem \ref{them1}, we consider the cases $0<J(v_0) <d$, $I(v_0) > 0$; $J(v_0)=d$, $I(v_0) >0$ separately.
We begin with the subcritical case, i.e., $0<J(v_0)<d$, $I(v_0)>0$. In this case we take advantage of the invariance of the family of potential wells described in Theorem \ref{invariant} to prove the exponential decay of solutions to \eqref{main}  in the $H^1_0$ norm.
\begin{theorem}\label{them1}
	Assume that  $I(v_0)>0$, $J(v_0)<d$ and $v$ is a a global solution to \eqref{main}.  Then  there exist constants $\mu >0$ and $C>0$  such that
	\begin{equation}\label{decay}
	\|v(\cdot,t)\|_{H^1_0({{U}})}\leq C e^{-\mu t}, \quad 0 \leq t < \infty.
	\end{equation}
\end{theorem}
\begin{proof}
	From the  proof of Theorem \ref{global_exis}, we get $I(v(\cdot, t))>0$, and  $0< J(v(\cdot, t))<d$. Therefore from Theorem \ref{invariant}, we deduce $I_\delta (v(\cdot, t)) >0$ for $0<\delta<1$. Set $\beta=1-\delta$ and observe that
	\begin{equation*}
	\int\limits_{{{U}}} |v(x,t)|^{p+1} \log |v(x,t)| dx<(1-\beta)\|\nabla v(\cdot, t)\|^2,
	\end{equation*}
	or
	\begin{equation}\label{loguu}
	\beta \|\nabla v(\cdot, t)\|^2 < I(v(\cdot, t)).
	\end{equation}
	On the other hand, it follows immediately that
	\begin{equation}\label{U_derivative}
	\frac{d}{dt} \|v(\cdot, t)\|_{H^1_0({{U}})}^2=-2I(v(\cdot, t)).
	\end{equation}
	Now using \eqref{loguu}--\eqref{U_derivative},  we obtain
	\begin{equation}\label{u_p}
	\begin{aligned}
	\frac{d}{dt} \|v(\cdot, t)\|^2_{H^1_0({{U}})}=-2I(v(\cdot, t))<-2 \beta \|\nabla v(\cdot, t)\|^2 \leq -2\beta  \frac{\lambda_1}{1+\lambda_1} \|v(\cdot, t)\|^2_{H^1_0({{U}})},
	\end{aligned}
	\end{equation}
	where $\lambda_1$ is the optimal  constant in the Poincar\'e inequality. Finally, Gronwall's lemma gives
	\begin{equation*}
	\|v(\cdot, t)\|_{H^1_0({{U}})} \leq  \|v_0\|_{H^1_0({{U}})} e^{-\mu t},\ t \geq 0,
	\end{equation*}
	where $\mu= {\beta  \frac{\lambda_1}{1+\lambda_1}}$. This completes the proof.
\end{proof}
We now turn our attention towards the critical case, i.e., $J(v_0)=d$, $I(v_0)>0$.
\begin{theorem}\label{them3}
	Assume   $I(v_0)\geq 0$, $J(v_0)=d$, and $v$ is a global solution to \eqref{main}. Then  there exist constants $C>0$, $\mu >0$ such that
	\begin{equation}\label{decay_2nd}
	\|v(\cdot,t)\|_{H^1_0({{U}})}\leq C e^{-\mu t}, \quad 0 \leq t < \infty.
	\end{equation}
\end{theorem}
\begin{proof}
	From the proof of Theorem \ref{global_exis}, it follows that  $I(v(\cdot,t)) \geq 0$ for $0 \leq t < \infty$. We complete the proof  by considering following two cases.\\
	\textbf{Case 1.}  Assume that $I(v(\cdot, t))>0$ for $t\geq 0$. Then from the relation
	$$(v_t, v)+(\nabla v_t, \nabla v)_2=-I(v(\cdot, t))<0,$$
 it follows that $\|v_t\|_{H^1_0({{U}})}>0$ and $\int_0^t \|v_t (\cdot, \tau)\|^2_{H^1_0({{U}})} d\tau$ is strictly increasing  in $[0,\infty)$. Therefore from  \eqref{energy_inequality}, we get
	\begin{equation}
	J(v(\cdot,t_1))=-\int_0^{t_1} \|v_t(\cdot, \tau)\|^2_{H^1_0({{U}})} d \tau+J(v_0)<d.
	\end{equation}
	Using the arguments that are employed in the proof of the decay estimate in Theorem  \ref{them1}, it is easy to obtain the exponential decay \eqref{decay_2nd}.\\
	\textbf{Case 2.} Let if possible there exists  $t_1>0$ such that $I(v(\cdot, t))>0$ for $0 \leq t< t_1$ and $I(v(\cdot,t_1))=0$. Now two possibilities can arise, they are: $(i)$ $\| \nabla v(\cdot,t_1)\|=0$, $(ii)$ $\| \nabla v(\cdot,t_1)\|> 0$.\\ We now prove that $\|\nabla v(\cdot,t_1)\|> 0$ can not hold. For, it is enough to show that if $\| \nabla v(\cdot,t_1)\|> 0$ then $I(v(\cdot, t_1))>0$.\\
	\textbf{Claim.} If $\|\nabla v(\cdot,t_1)\| >0$ then $I(v(\cdot,t_1))>0$.\\
	For $0\leq t<t_1$, since $(v_t, v)+(\nabla v_t, \nabla v)_2=-I(v(\cdot, t))$, it follows that $t \mapsto \int_0^t \|v_t\|^2_{H^1_0({{U}})} dt$ is strictly increasing.  Owing to \eqref{energy_inequality}, we get
	\begin{equation}\label{new12}
	J(v(\cdot,t_1))=-\int_0^{t_1} \|v_t(\cdot, \tau)\|^2_{H^1_0({{U}})} d \tau+J(v_0)<d.
	\end{equation}
	Since $\|\nabla v(\cdot,t_1)\| >0$ and $I(v(\cdot,t_1))=0$, from the definition of $d$ we get $J(v(\cdot,t_1)) \geq d$, which is contradiction to  \eqref{new12}. This proves Claim.\\
	Therefore we have $\|\nabla v(\cdot, t_1)\| =0$. Hence one can easily deduce that $v$ satisfies   \eqref{decay_2nd}.
\end{proof}
\section{Finite time blowup}
In this section, we prove some blowup results when $I(v_0)<0$. In fact, we use different tools to prove that the solutions exhibits finite time blow up depending on the value of $J(v_0)$.
\begin{theorem}\thlabel{them2}
	Let $v_0 \in H_0^1({{U}})$. Assume  $I(v_0)<0$ and $J(v_0)<d$, then the weak solution to problem \eqref{main} blows up in finite time, $i.e.$, there exists $ T>0$ such that $$\lim\limits_{t \to T^{-}} \|v\|_{H^1_0({{U}})}=\infty.$$
\end{theorem}
\begin{proof}
	Let $v(x,t)$ be any solution to problem \eqref{main} with $I(v_0)<0$ and $J(v_0)<d$.\\
	Define the function $N \colon [0,\infty) \to \mathbb{R}^+$  by $N(t)=\int\limits_0^t \|v\|_{H^1_0({{U}})}^2 d\tau $. Then an easy computation yields 
	\begin{equation}\label{m_derivative}
		\dot{N}(t)=\|v\|_{H^1_0({{U}})}^2,\ \ddot{N}(t)=-2I(v).
	\end{equation}
	From \eqref{combo},  \eqref{energy_inequality}, and the Poincar\'{e} inequality  there exists $ \lambda_1 >0$ such that 
	\begin{equation}\label{mdd}
		\begin{aligned}
			\ddot{N}(t)&=-2(1+p)J(v)+(p-1)\|\nabla v\|^2+\frac{2}{1+p}\|v\|^{1+p}_{1+p}
			\\
			&\geq 2(1+p)\int\limits_0^t \|v_\tau\|^2 d \tau +(p-1)\frac{\lambda_1}{\lambda_1+1}\dot{N}(t)-2(1+p)J(v_0).
		\end{aligned}
	\end{equation}
	Since 
	\begin{equation}\label{mint}
		\begin{aligned}
			\left(\int\limits_0^t (v_t, v)_{H^1_0({{U}})} d \tau \right)^2&=\left(\frac{1}{2}\int\limits_0^t \frac{d}{d t} \|v\|^2_{H^1_0({{U}})} d\tau \right)^2
			\\
			&=\frac{1}{4}\left( \dot{N}^2(t)-2\|v_0\|^2_{H^1_0({{U}})} \dot{N}(t)+\|v_0\|_{H^1_0({{U}})}^4\right),
		\end{aligned}
	\end{equation}
	we obtain 
	\begin{equation*}
		\begin{aligned}
			N\ddot{N}- \Big(\frac{1+p}{2}\Big)\dot{N}^2 &\geq (p-1)\frac{\lambda_1}{\lambda_1+1} N \dot{N}-2(1+p)J(v_0)N-(1+p)\|v_0\|^2_{H^1_0({{U}})} \dot{N}+\frac{1+p}{2}\|v_0\|^4_{H^1_0({{U}})}
			\\
			&\quad \quad + 2(1+p) \left\{ \int\limits_0^t \|v\|^2_{H^1_0({{U}})} d \tau \int\limits_0^t \|v_t\|^2_{H^1_0({{U}})} d \tau -\left(\int\limits_0^t (v_t, v)_{H^1_0({{U}})} d \tau\right)^2\right\}.
		\end{aligned}
	\end{equation*}
	By H\"older's inequality, we get 
	\begin{equation}\label{bode}
		\begin{aligned}
			N\ddot{N}- \frac{1+p}{2}\dot{N}^2 &\geq (p-1)\frac{\lambda_1}{\lambda_1+1}N \dot{N}
			-2(1+p)J(v_0)N-(1+p)\|v_0\|^2_{H^1_0({{U}})} \dot{N}.
		\end{aligned}
	\end{equation}
	\textbf{Claim:} For large $t$, it follows that
	\begin{equation}\label{ode}
	N\ddot{N}- \Big(\frac{1+p}{2}\Big)\dot{N}^2>0.
	\end{equation}
	To prove this claim we consider two cases and discuss separately.\\
	\textbf{Case-1:} Assume  $J(v_0)\leq 0.$ From \eqref{m_derivative} and \eqref{mdd}, we get 
	$\ddot{N} \geq 0$. Since, $\dot{N}(t)=\|v\|^2_{H^1_0({{U}})}\geq 0$, then  there exists $t_0 \geq 0$ such that $\dot{N}(t_0)>0$ and 
	$$N(t)\geq {N}(t_0)+\dot{N}(t_0)(t-t_0)> \dot{N}(t_0)(t-t_0), \quad t\geq t_0.$$
	Thus for $t$ large, we have $(p-1)\lambda_1 N>(1+p)\|v_0\|_{H^1_0({{U}})}^2$ and \eqref{ode} holds.
	\\
	\textbf{Case-2:} Assume that $0<J(v_0)<d.$ 
	From Theorem \ref{invariant}, we get $v(\cdot,t) \in V_\delta $ for $1\leq \delta<\delta_2$ and $t>0$, where $\delta_2$ is the same as the one introduced in  Theorem \ref{invariant}. Thus $I_{\delta}(v(\cdot,t))<0$,  for $1\leq \delta<\delta_2,$ $ t \geq 0$. Next we prove  that $\|\nabla v \|^2 > \frac{\lambda_1}{1+\lambda_1}\|v_0\|^2_{H^1_0({{U}})}>0,\ t \geq 0$. For, since $\frac{d}{dt}\|v\|_{H^1_0({{U}})}^2=-2I(v)>0$, we obtain $t \mapsto \|v\|_{H^1_0({{U}})}^2, \ t \geq 0$ is a strictly increasing function. On the other hand, the Poincar\'e inequality  gives 
	\medskip \\
	\centerline{$\|\nabla v\|^2 \geq \frac{\lambda_1}{1+\lambda_1} \|v\|_{H^1_0({{U}})}^2> \frac{\lambda_1}{1+\lambda_1} \|v_0\|_{H^1_0({{U}})}^2.$}
	\medskip \\
	Therefore it is easy to get that $\|\nabla v \|^2 > \frac{\lambda_1}{1+\lambda_1}\|v_0\|^2_{H^1_0({{U}})}>0,\ t \geq 0$. Now from  \eqref{m_derivative} and the definition of $I_{\delta_2}$, we find that 
	\begin{equation*}
		\ddot{N}(t)\geq -2I_{\delta_2}(v)+2(\delta_2-1)\|\nabla v\|^2\geq2(\delta_2-1)r^2(\delta_2)> 0,
	\end{equation*}
	\begin{equation*}
		\dot{N}(t)\geq \dot{N}(0) +2(\delta_2-1)r^2(\delta_2)t\geq  2(\delta_2-1)r^2(\delta_2)t,
	\end{equation*}
	and 
	\begin{equation*}
		N(t)\geq {N}(0)+(\delta_2-1)r^2(\delta_2)t^2\geq  (\delta_2-1)r^2(\delta_2)t^2.
	\end{equation*}
	Thus for $t$ large, we have
	\begin{equation}\label{Ode_inequa}
		\left\{
		\begin{aligned}
			&(p-1)\frac{\lambda_1}{\lambda_1+1} N(t)>2(1+p)\|v_0\|_{H^1_0({{U}})}^2,
			\\
			&(p-1)\frac{\lambda_1}{\lambda_1+1}\dot{N}(t)>4(1+p)J(v_0).
		\end{aligned}
		\right.
	\end{equation}
	On substituting \eqref{Ode_inequa} in  \eqref{bode}, we get \eqref{ode} for sufficiently large $t$.\\
	On the other hand, a straightforward calculation gives
	\begin{equation*}
		\left(N^{-\alpha}\right)^{\prime\prime}	=-\alpha N^{-\alpha-2} \left(N\dot{N}-(\alpha+1)\dot{N}^2\right).
	\end{equation*}
	For $\alpha= \frac{p-1}{2}$, \eqref{ode} implies $\left(N^{-\alpha}\right)^{\prime\prime}<0$ for sufficiently large $t>0$. Hence, for $t >\tilde{t}$, we can write 
	\begin{equation*}
		N^{-\frac{p-1}{2}}(t) < N^{-\frac{p-1}{2}}(\tilde{t})\left(1-\left(\frac{p-1}{2}\right)\frac{\dot{N}(\tilde{t})}{N(\tilde{t})}(t-\tilde{t})\right),
	\end{equation*} 
	which implies that  there exists $ T>0$ such that 
	$$\lim\limits_{t \to T^{-}}N^{-\frac{p-1}{2}}(t)=0.$$
	This completes the proof.
\end{proof}
We now focus on the critical case, i.e., $J(v_0)=d$, $I(v_0) <0$.
\begin{theorem}\thlabel{them4}
	Let $v_0 \in H_0^1({{U}})$. Assume that  $I(v_0)<0$ and $J(v_0)=d$, then the weak solution $v$ to equation \eqref{main} blows up in finite time.
\end{theorem}
\begin{proof}
	Assume that  $T$ is the existence time of $v$. We need to prove that $T <\infty$. From the continuity of $I(v)$ and  $J(v)$, there exists $t_1 \in (0,T)$ small enough such that $I(v(\cdot,t))<0$ and $J(v(\cdot,t))>0$ for $t\in [0,t_1]$. Therefore $\int\limits_0^t \|v_\tau\|_{H^1_0({{U}})}^2 d \tau $ is strictly increasing for $t\in [0,t_1]$. From \eqref{energy_inequality}, we can choose $t_1$ such that  
	\begin{equation}
		0<J(v(\cdot,t_1))=J(v_0)-\int\limits_0^{t_1} \|v_t\|_{H^1_0({{U}})}^2 d \tau<J(v_0)=d.
	\end{equation}
	If we take $t_1$ as the initial time in \thref{them2}, it follows that the maximal existence time $T$ of $v$ is finite, $i.e.,$ 
	$$\lim\limits_{t \to T^{-}} \|v(t)\|_{H^1_0({{U}})}=\infty.$$
	This completes the proof. 
\end{proof}
\noindent Finally, we discuss the finite time blow-up property when the initial energy is high.
\begin{theorem}
	Let $v_0 \in H^1_0({{U}})$. Moreover assume that the initial data satisfies\\
	$(i)$ $J(v_0)>0$, \medskip \\
	$(ii)$ $\|v_0\|_{H^1_0({{U}})}^2>\frac{2{(\lambda_1+1)}(1+p)}{\lambda_1(p-1)} J(v_0)$,\medskip \\
	$(iii)$ $I(v_0)<0$, \medskip \\
	where $\lambda_1$ is the  optimal  constant in the Poincar\'e inequality. Then the weak solution $v$ to equation \eqref{main} blows up in finite time.
\end{theorem}
\begin{proof}
	As in the proof of \thref{them2}, we work with the quantity $N(t)=\int\limits_0^t \|v\|_{H^1_0({{U}})}^2 d\tau $. We prove the theorem in two steps: $(i)$ $I(v)<0$ and $\| v(t)\|_{H^1_0({{U}})}^2 > \frac{2({\lambda_1+1})(1+p)}{\lambda_1(p-1)} J(v_0)$, for every $t \in (0, T)$, $(ii)$  for sufficiently large $t$, $\ddot{N}N-\frac{1+p}{2}\dot{N}^2>0$.\\
	\textbf{Step 1.} Let if possible, there exists $t_0 \in (0,T)$ such that $I(v(\cdot,t_0)) =0$ and $I(v(\cdot,t))<0$ for $0\leq t <t_0$. Again consider the function $N(t)$ as in  \thref{them2}. Since  $\ddot{N}(t)= - 2I(v)>0$ for $t\in [0,t_0)$,  $\dot{N}$ is increasing in $[0,t_0]$, we deduce that
	\begin{equation}\label{mt0}
		\dot{N}(t_0)>\dot{N}(0)=\|v_0\|_{H^1_0({{U}})}^2> \frac{2({\lambda_1+1})(1+p)}{\lambda_1(p-1)} J(v_0).
	\end{equation}
	Since $I(v(\cdot,t_0))=0$, the conservation of energy  gives 
	\begin{equation}
		\begin{aligned}
			J(v_0)&\geq J(v(\cdot,t_0))
			\\
			&=\frac{p-1}{2(1+p)}\|\nabla v(\cdot,t_0)\|^2+ \frac{1}{(1+p)^2}\|v(\cdot,t_0)\|_{1+p}^{1+p}
			\\
			&\geq \frac{p-1}{2(1+p)}\|\nabla v(\cdot,t_0)\|^2
			\\
			&\geq  \frac{\lambda_1 (p-1)}{2({\lambda_1+1})(1+p)}\| v(\cdot,t_0)\|_{H^1_0({{U}})}^2.
		\end{aligned}
	\end{equation}
	Hence we get $\dot{N}(t_0)=\|v (\cdot,t_0)\|^2 \leq \frac{2({\lambda_1+1})(1+p)}{\lambda_1 (p-1)} J(v_0)$, which is a contradiction to \eqref{mt0}. Thus we obtain
	\begin{equation*}
		I(v)<0, \  t\in (0,T),
	\end{equation*}
	and \eqref{mt0} holds for every $t \in (0,T)$.
	Hence $N$ is strictly increasing. Therefore for large $t$, we obtain
	\begin{equation}\label{mm}
		N(t)>\frac{2({\lambda_1+1})(1+p)}{\lambda_1 (p-1)} \|v_0\|_{H^1_0({{U}})}^2.
	\end{equation}
	\textbf{Step 2.} From \eqref{bode}, \eqref{mt0} and \eqref{mm}, we deduce that for sufficiently large $t$, 
	\begin{equation}
		\begin{aligned}
			\ddot{N}N-\frac{1+p}{2}\dot{N}^2
			&\geq (p-1)\frac{\lambda_1}{\lambda_1+1} N \dot{N}
			-2(1+p)J(v_0)N-(1+p)\|v_0\|_{H^1_0({{U}})}^2 \dot{N} >0.
		\end{aligned}
	\end{equation}
	Using the arguments employed in \thref{them2}, we conclude the proof of the theorem.
\end{proof}

\section*{Acknowledgements}
\noindent The first author would like to thank CSIR (award number: 09/414(1154)/2017-EMR-I) for providing the financial support for his research.
The third author is supported by Department of Science and Technology, India, under MATRICS (MTR/2019/000848).
\bibliographystyle{abbrv}
\bibliography{references}
\end{document}